\documentclass[reqno,a4paper]{amsart}

\usepackage[ngerman,USenglish]{babel}
\usepackage[ascii]{inputenc}
\usepackage{eucal}
\usepackage{color}
\usepackage{epsfig}
\usepackage{hyperref}

\allowdisplaybreaks[4]

\newcommand{\arxiv}[1]{\href{http://arxiv.org/abs/#1}{\nolinkurl{arXiv:#1}}}
\newcommand{\mr}[1]{\href{http://www.ams.org/mathscinet-getitem?mr=MR#1}{MR#1}}
\newcommand{\zbl}[1]{\href{http://www.zentralblatt-math.org/zmath/en/advanced/?q=an:#1&format=complete}{Zbl \nolinkurl{#1}}}
\newcommand{\doi}[1]{\href{http://dx.doi.org/#1}{\nolinkurl{doi:#1}}}

\newtheorem{theorem}{Theorem}
\newtheorem{corollary}[theorem]{Corollary}

\author{Thomas Bliem}

\address{\foreignlanguage{ngerman}{Mathematisches Institut,
Universit\"at zu K\"oln,
Weyertal 86--90,
50931 K\"oln},
Germany}

\date{July 2008}

\thanks{Supported by the European Union (Marie Curie Research Training Network ``Liegrits'') and by the Deutsche Forschungsgemeinschaft (Sonderforschungsbereich/Transregio 12).}

\thanks{This is a preprint version, the original article appeared on
pp.\ 80--86 in M.\ Dehmer, M.\ Drmota, F.\ Emmert-Streib (ed.), \emph{Proceedings of the 2008 international conference on information theory and statistical learning}, CSREA Press, 2008.}

\title{Weight multiplicities for $\mathfrak{so}_5(\mathbf{C})$}

\hypersetup{%
pdftitle={Weight multiplicities for so5(C)},%
pdfkeywords={Laplace transformation, vector partition function, weight multiplicities, Jeffrey-Kirwan residue},%
pdfauthor={Thomas Bliem},%
bookmarksnumbered=true}


\begin{document}

\maketitle

\begin{abstract}
We explicitly determine quasi-polynomials describing the weight multiplicities of the Lie algebra $\mathfrak{so}_5(\mathbf{C})$.
This information entails immediate complete knowledge of the character of any simple representation as well as the asymptotic behavior of characters.
\end{abstract}

\section{Introduction}

There have been investigations into the characters of simple Lie algebras initiated in the beginning of the discipline and going on to present days.
From the early days it was possible to write down formulas explicitly describing the characters, as notably done by H.~Weyl \cite{weyl1925f} and B.~Kostant \cite{kostant1959}.
P.~Littelmann's path model \cite{littelmann1994, littelmann1995} does not formally furnish a formula but rather an algorithm allowing to calculate the character performing finitely many combinatorial operations, so fits into the very same context.

Still, all these approaches, while allowing to calculate characters or individual weight multiplicities, at least in principle and for small instances, do not fully exhibit the rich structure underlying the characters.
For example, from G.~Heckman's thesis \cite{heckman1980} it is known that considering a sequence of simple representations for a given simple Lie algebra such that the highest weights of the elements are the integral multiples of a given weight, the corresponding characters show a particular behavior of convergence.

As pointed out by S.~Billey, V.~Guillemin and E.~Rassart \cite{billey2004}, for the case of $\mathfrak{sl}_k(\mathbf{C})$ Gelfand-Tsetin patterns \cite{gelfand1950} can be used as a key ingredient to develop descriptions of characters better reflecting their structure.
In the following I will demonstrate that it is possible, substituting Gelfand-Tsetlin patterns by P.~Littelmann's patterns \cite{littelmann1998}, using B.~Sturmfels' structure theorem \cite{sturmfels1995} on vector partition functions and Laplace transformation methods developed by L.~Jeffrey, F.~Kirwan \cite{jeffrey1995}, A.~Szenes and M.~Vergne \cite{szenes2003} as well as work by C.~De Concini and C.~Procesi on the combinatorics of residues \cite{deconcini2005}, to obtain indeed complete knowledge, structural and computational, of the characters of $\mathfrak{so}_5(\mathbf{C})$.
In fact, I do not use any special properties of the Lie algebra $\mathfrak{so}_5(\mathbf{C})$.
This is just a random example picked to demonstrate the power of the combination of the above-mentioned ideas, which, in principle, are applicable to any semisimple complex Lie algebra.

\section{Preliminaries}

Consider the Lie algebra $\mathfrak{so}_5(\mathbf{C})$ of complex $(5 \times 5)$-matrices $A$ such that $A^t M = - M A$ for a fixed nondegenerate complex $(5 \times 5)$-matrix $M$.
Choose a Cartan subalgebra $\mathfrak{h}$ and simple roots $\alpha_1, \alpha_2 \in \mathfrak{h}^*$ such that $\alpha_2$ is the long root.
The Dynkin diagram associated to this enumeration of the simple roots is \[ 1 \Longleftarrow 2.\]
Let $\omega_1, \omega_2 \in \mathfrak{h}^*$ be the corresponding fundamental weights, $\Lambda \subseteq Q \subseteq \mathfrak{h}^*$ the weight lattice respectively the root lattice.
The irreducible $\mathfrak{so}_5(\mathbf{C})$-module of highest weight $\lambda$ is denoted by $V(\lambda)$.
For a weight $\mu$ denote by $V(\lambda)_\mu$ the space of vectors of weight $\mu$ in $V(\lambda)$.
If $\lambda - \mu$ is not an element of the root lattice, then $V(\lambda)_\mu = 0$.
For $\lambda$ a dominant weight and $\beta$ an element of the root lattice, define \[ K^\lambda_\beta := \dim V(\lambda)_{\lambda - \beta}. \]
If the weight $\lambda$ is not dominant we define all $K^\lambda_\beta$ to be $0$.
The character of $V(\lambda)$ is by definition the element \[ \sum_{\beta \in Q} K^\lambda_\beta \cdot [\lambda - \beta] \] of the group ring $\mathbf{Z}[\Lambda]$.
Hence knowledge of all the characters of $\mathfrak{so}_5(\mathbf{C})$ is equivalent to knowledge of the function $K : \Lambda \times Q \to \mathbf{Z_{\geq 0}}$.

For any dominant weight $\lambda$, let $B(\lambda)$ be the crystal associated to $V(\lambda)$ by M.~Kashiwara \cite{kashiwara1995} and $B(\infty) := \varinjlim_\lambda B(\lambda) \otimes T_{-\lambda}$ the direct limit of the crystals $B(\lambda)$, shifted to have highest weight $0$.
Consider the reduced decomposition $w_0 = s_1 s_2 s_1 s_2$ of the longest element of the Weyl group.
To this decomposition, P.~Littelmann \cite{littelmann1998} associates a convex polyhedral cone $\mathcal{C} \subseteq \mathbf{R}^4$, a family of polytopes $\mathcal{C^\lambda} \subseteq \mathcal{C}$ for dominant weights $\lambda$ and a $\mathbf{Z}$-linear map $\psi: \mathbf{Z}^4 \to Q$ such that:
(i) There is a canonical bijection $\sigma : B(\infty) \to \mathcal{S} := \mathcal{C} \cap \mathbf{Z}^4$.
(ii) For each dominant weight $\lambda$, the bijection $\sigma$ restricts to a bijection between $B(\lambda)$ and $\mathcal{S}^\lambda := \mathcal{C}^\lambda \cap \mathbf{Z}^4$.
(iii) The weight of any element $b \in B(\lambda)$ is $\mathrm{wt}(b) = \lambda - \psi(\sigma(b))$.

Specifically, denote the standard coordinates by $a_{22}$, $a_{11}$, $a_{12}$, $a_{13}$.
Then the cone $\mathcal{C}$ is given by the inequalities \begin{equation} \label{eq:C} 2a_{11} \geq a_{12} \geq 2a_{13} \geq 0, \quad a_{22} \geq 0 . \end{equation}
For a dominant weight $\lambda = \lambda_1 \omega_1 + \lambda_2 \omega_2$, the polytope $\mathcal{C}^\lambda$ is given inside $\mathcal{C}$ by the additional inequalities
\begin{equation}  \label{eq:C^lambda}\begin{split}
a_{13} &\leq \lambda_2, \\
a_{12} &\leq \lambda_1 + 2 a_{13}, \\
a_{11} &\leq \lambda_2 + a_{12} - 2 a_{13}, \\
a_{22} &\leq \lambda_1 + 2 a_{11} - 2 a_{12} + 2 a_{13}.
\end{split} \end{equation}
The $\mathbf{Z}$-linear map is $\psi = (a_{22}+ a_{12}) \alpha_1 + (a_{11} + a_{13}) \alpha_2$.
For a given dominant weight $\lambda$ and any element $\beta = \beta_1 \alpha_1 + \beta_2 \alpha_2$ of the root lattice, define \begin{equation} \label{eq:C^lambda_beta} \mathcal{C}^\lambda_\beta := \{ a \in \mathcal{C}^\lambda : a_{22}+ a_{12} = \beta_1,\ a_{11} + a_{13} = \beta_2 \} \end{equation} and $\mathcal{S}^\lambda_\beta := \mathcal{C}^\lambda_\beta \cap \mathbf{Z}^4$.
Then by \cite{littelmann1998} we have $K^\lambda_\beta = \lvert \mathcal{S}^\lambda_\beta \rvert$, that is: Determination of the weight multiplicities is reduced to counting the number of points in polytopes.

\section{Reformulation using vector partition functions}

For any positive integers $n$ and $N$ and any $(n \times N)$-matrix $A$ with integral coefficients such that $\ker A \cap \mathbf{R}_{\geq 0}^n = 0$ define the vector partition function $\Phi_A : \mathbf{Z}^n \to \mathbf{Z}_{\geq 0}$ by $\Phi_A(v) := \lvert \{ a \in \mathbf{Z}_{\geq 0}^N : A a = v \} \rvert$.
We will now reformulate Littelmann's result and explicitly determine matrices $A$ and $B$ such that $K^\lambda_\beta = \Phi_A(B \cdot (\lambda_1, \lambda_2, \beta_1, \beta_2)^t)$.
Indeed, the inequalities \eqref{eq:C} and \eqref{eq:C^lambda} can be turned into equations using slack variables $s_1, s_2$ respectively $t_1, t_2, t_3, t_4$.
Hence, the number of integral solutions $\lvert \mathcal{S}^\lambda_\beta \rvert$ of the system (\ref{eq:C}, \ref{eq:C^lambda}, \ref{eq:C^lambda_beta}) is equal to the number of nonnegative integral solutions to the system
\[ \begin{array}{*6{r@{}}l}
& 2a_{11} & {}-a_{12} & & {}-s_1 & & =0, \\
& & a_{12} & {}-2a_{13} & {}-s_2 & & =0, \\
& & & a_{13} & & {}+t_1 & =\lambda_2, \\
& & a_{12} & {}-2a_{13} & & {}+t_2 & =\lambda_1, \\
& a_{11} & {}-a_{12} & {}+2a_{13} & & {}+t_3 & =\lambda_2, \\
a_{22} & {}-2a_{11} & {}+2a_{12} & {}-2a_{13} & & {}+t_4 & =\lambda_1, \\
a_{22} & & {}+a_{12} & & & & =\beta_1, \\
& a_{11} & & {}+a_{13} & & & =\beta_2.
\end{array} \]
In other words, $\lvert \mathcal{S}^\lambda_\beta \rvert = \Phi_A(B \cdot (\lambda_1, \lambda_2, \beta_1, \beta_2)^t)$ for matrices
\begin{equation} \label{eq:E} A = \left( \begin{array}{cccc|cc|cccc}
 0 &  2 & -1 &  0 & -1 &  0 &  0 &  0 &  0 &  0 \\
 0 &  0 &  1 & -2 &  0 & -1 &  0 &  0 &  0 &  0 \\ \hline
 0 &  0 &  0 &  1 &  0 &  0 &  1 &  0 &  0 &  0 \\
 0 &  0 &  1 & -2 &  0 &  0 &  0 &  1 &  0 &  0 \\
 0 &  1 & -1 &  2 &  0 &  0 &  0 &  0 &  1 &  0 \\
 1 & -2 &  2 & -2 &  0 &  0 &  0 &  0 &  0 &  1 \\ \hline
 1 &  0 &  1 &  0 &  0 &  0 &  0 &  0 &  0 &  0 \\
 0 &  1 &  0 &  1 &  0 &  0 &  0 &  0 &  0 &  0
\end{array} \right) \end{equation}
and
\begin{equation} \label{eq:B} B = \left( \begin{array}{cc|cc}
 0 & 0 & 0 & 0 \\
 0 & 0 & 0 & 0 \\ \hline
 0 & 1 & 0 & 0 \\
 1 & 0 & 0 & 0 \\
 0 & 1 & 0 & 0 \\
 1 & 0 & 0 & 0 \\ \hline
 0 & 0 & 1 & 0 \\
 0 & 0 & 0 & 1
\end{array} \right) . \end{equation}

\section{Structure and calculation of vector partition functions}
\label{sect:vpf}

The presentation of $K^\lambda_\beta$ in terms of a vector partition function gains its strength from the following structure theorem of B.~Sturmfels \cite{sturmfels1995}:

\begin{theorem} \label{th:sturmfels}
Let $A \in \mathbf{Z}^{(n, N)}$ such that the vector partition function $\Phi_A : \mathbf{Z}^n \to \mathbf{Z}_{\geq 0}$ is defined.
Then there is a homogeneous fan $F = \mathrm{fan}(A)$ in $\mathbf{R}^n$ and a family of quasi-polynomials $(f_C)$ on $\mathbf{Z}^n$, indexed by the maximal cones in $F$, such that $\Phi_A$ coincides with $f_C$ on $C \cap \mathbf{Z}^n$ and vanishes outside the support of $F$.
\end{theorem}

Here, a \emph{fan} is a finite set of convex polyhedral cones, closed under taking faces, and such that the intersection of any two cones is a face of both.
The fan being \emph{homogeneous} means that all maximal cones have the same dimension, which is in this case the rank of $A$.
A function $f : \mathbf{Z}^n \to \mathbf{C}$ is called a \emph{quasi-polynomial} if there is a lattice $L \subseteq \mathbf{Z}^n$ and a family $(f_{\bar{h}})_{\bar{h} \in \mathbf{Z}^n / L}$ of polynomials on $\mathbf{Z}^n$ such that $f(h) = f_{\bar h}(h)$ for all $h \in \mathbf{Z}^n$.

While the naive algorithm for computing individual values $\Phi_A(v)$ of a given vector partition functions has exponential execution time with respect to the \emph{components} of $v$, this theorem allows the following strategy:
Determine the maximal cones of $F$ and for each maximal cone $C$ determine the quasi-polynomial $f_C$.
This task being accomplished, individual values $\Phi_A(v)$ of the vector partition function can be calculated by evaluating the corresponding quasi-polynomial at $v$, the execution time of which is of order of the \emph{logarithm} of the components of $v$.
In this sense, determination of the maximal cones and quasi-polynomials yields instant complete knowledge of the values of the given vector partition function.
We will indeed perform these steps for the vector partition function given by \eqref{eq:E}.
In this way we will determine all the characters of $\mathfrak{so}_5(\mathbf{C})$ at once.

Generally, the maximal cones and quasi-polynomials of a vector partition function can be determined as outlined in the sequel.
For a more extensive treatment, refer to the study of the Kostant partition function for classical root systems \cite{baldoni2006} by W.~Baldoni, M.~Beck, Ch.~Cochet and M.~Vergne, which also served as a model for the following calculations.

Suppose that $A$ has rank $n$.
For a lattice $L \subseteq \mathbf{R}^n$ denote by $L^\perp \subseteq (\mathbf{R}^n)^*$ the corresponding dual lattice.
If $L \subseteq \mathbf{Z}^n$ then $L^\perp \supseteq  (\mathbf{Z}^n)^\perp$.
Let $T := (\mathbf{R}^n)^* / (\mathbf{Z}^n)^\perp$, an $n$-dimensional torus.
Denote the column vectors of $A$ by $a_1, \ldots, a_N$.
A subset $\sigma \subseteq \{ 1, \ldots, N \}$ is called a \emph{basic subset} for $A$ if $(a_i)_{i \in \sigma}$ is a basis of $\mathbf{R}^n$.
For any basic subset $\sigma = \{ i_1, \ldots, i_n \} $ let $T(\sigma) := (\mathbf{Z}a_{i_1} + \cdots + \mathbf{Z}a_{i_n})^\perp \bmod (\mathbf{Z}^n)^\perp \subseteq T$.
The set $T(\sigma)$ contains $\mathrm{vol}(\sigma) := \left| \det(a_i : i \in \sigma) \right|$ elements.
Let $\Gamma \subseteq T$ be the union of all $T(\sigma)$ for basic subsets $\sigma$ for $A$.

The fan $F = \mathrm{fan}(A)$ associated to $A$ can be described as follows:
For any basic subset $\sigma$ denote by $\mathrm{cone}(\sigma)$ the convex polyhedral cone generated by $\{a_i : i \in \sigma\}$.
Cones of the form $\mathrm{cone}(\sigma)$ are called \emph{basic cones}.
Then the maximal cones of $F$ are the minimal $n$-dimensional cones which can be written as an intersection of basic cones.

For $h \in \mathbf{Z}^n$ and $g \in T$ define the \emph{Kostant function} as the meromorphic function on $(\mathbf{R}^n)^* \otimes_\mathbf{R} \mathbf{C} = (\mathbf{C}^n)^*$ given by
\[ F_{g,h}(u) := \frac{e^{\langle u + 2 \pi i g, h \rangle}}{\prod_{k=0}^N \left( 1 - e^{-\langle u + 2 \pi i  g, a_k \rangle } \right) } . \]
Note that $\langle g, h \rangle$ and $\langle g, a_k \rangle$ are determined modulo $\mathbf{Z}$, so the values of the exponential functions are unambiguous.

A basic subset $\{ i_1, \ldots, i_n \}$ with $ i_1 < \cdots < i_n $ is called \emph{without broken circuits} if there are no $j \in \{1, \ldots, n \}$ and $k \in \{ i_j+1, \ldots, N \}$ such that the family $(a_{i_1}, \ldots, a_{i_j}, a_k)$ is linearily dependent.\footnote{Note that De Concini and Procesi \cite{deconcini2005} use the inverse ordering.}
For a maximal cone $C$ of the fan $F$, let $B_\mathrm{nb}(C)$ denote the set of basic subsets $\sigma$ without broken circuits such that $\mathrm{cone}(\sigma) \supseteq C$.

For a meromorphic function $f(u)$ on $(\mathbf{R}^n)^* \otimes_\mathbf{R} \mathbf{C} = (\mathbf{C}^n)^*$ with poles along $a_i^\perp\ (i = 1, \ldots, N)$ and a basic subset $\sigma$, define the \emph{iterated residue} of $f$ with respect to $\sigma$ to be \[ \operatorname{ires}_\sigma f(u) := \operatorname{res}_{a_{i_n}=0} \cdots \operatorname{res}_{a_{i_1}=0} f(u) , \] where $a_{i_j}$ are interpreted as coordinates on $(\mathbf{C}^n)^*$.

The following theorem is a combination of A.\ Szenes and M.~Vergne's expression \cite[th.\ 3.1]{szenes2003} of a vector partition function as a Jeffrey-Kirwan residue and C.~De Concini and C.~Procesi's work \cite{deconcini2005} on the Jeffrey-Kirwan residue:

\begin{theorem}
On any maximal cone $C$ of $\mathrm{fan}(A)$, the vector partition function associated to $A$ is given by \[ \Phi_A(h) = \sum_{\sigma \in B_\mathrm{nb}(C)} \frac{1}{\mathrm{vol}(\sigma)} \sum_{g \in \Gamma} \operatorname{ires}_\sigma F_{g, h}(u) . \]
\end{theorem}

\section{Computing the weight multiplicity function for $\mathfrak{so}_5(\mathbf{C})$}

Recall that $K^\lambda_\beta = \Phi_A(B \cdot (\lambda_1, \lambda_2, \beta_1, \beta_2)^t)$ for $\lambda = \lambda_1 \omega_1 + \lambda_2 \omega_2$, $\beta = \beta_1 \alpha_1+ \beta_2 \alpha_2$ and $A$, $B$ matrices as given in equations \eqref{eq:E} respectively \eqref{eq:B}.
We are now ready to perform explicit computations.
All computations were performed using Maple 11 by Maplesoft\footnote{\url{http://www.maplesoft.com/}} and the package Convex 1.1.2 by M.~Franz\footnote{\url{http://www-fourier.ujf-grenoble.fr/~franz/}}.

First we have to determine $F = \mathrm{fan}(A)$.
This is done as follows:
Any maximal cone in $F$ is the intersection of all the basic cones containing it.
We can hence find the neighbors of a given maximal cone $C$ as follows:
For each facet $f$ of $C$, the neighboring maximal cone of $C$ in direction $f$ is the intersection of all basic cones $\mathrm{cone}(\sigma)$ such that $f \subseteq \mathrm{cone}(\sigma)$ and $(C \subseteq\mathrm{cone}(\sigma) \implies f \not\subseteq \partial(\mathrm{cone}(\sigma)))$.
So we start with an arbitrary maximal cone and find the others by a standard algorithm for graph traversal using this description of the neighbor relation.
There are 320 maximal cones alltogether.

In order to determine $\Gamma$ we proceed as follows:
For any basic subset $\sigma$, let $A_\sigma$ be the submatrix of $A$ consisting of the columns with indices in $\sigma$.
The subgroup $T(\sigma)$ of $T$ is generated by the classes of the row vectors of $A_\sigma^{-1}$.
So we start with the set of these classes and determine its closure under the operation of adding the class of any row vector of $A_\sigma^{-1}$ by a standard algorithm of graph traversal.

The set of basic subsets without broken circuits is determined straightforwardly using the definition.
In order to speed up the calculation, basic subsets are built up recursively, checking the additional prerequisites at every step of the recursion.

In fact we are only interested in maximal cones whose intersection with the image of $B$ has dimension $4$.
There are 43 such intersections.
For the calculation of the quasi-polynomials we now pick for each such intersection $\mathfrak{c}$ a maximal cone $C$ of $\mathrm{fan}(A)$  such that $\mathfrak{c} = C \cap \mathrm{im}(B)$.
Then we compute the quasi-polynomials for each of these maximal cones as described in section \ref{sect:vpf}.
The quasi-polynomials coincide for some of the neighboring $\mathfrak{c}$, so we glue together the corresponding cones.
The preimage under $B$ of the resulting fan is given by the following maximal cones in $\mathbf{R^4} = (\Lambda \otimes_\mathbf{Z} \mathbf{R}) \times (Q \otimes_\mathbf{Z} \mathbf{R})$:

{
\begin{align*}
\mathfrak{c}_{1} &= \parbox[t]{0.8\columnwidth}{$\{\lambda_{{2}}-\beta_{{2}}
 \geq 0,\ \lambda_{{1}}-\beta_{{1}}
 \geq 0,\ \beta_{{1}}-\beta_{{2}}
 \geq 0,\ -\beta_{{1}}+2\beta_{{2}}
 \geq 0\},$}
\\
\mathfrak{c}_{2} &= \parbox[t]{0.8\columnwidth}{$\{\beta_{{1}}-2\beta_{{2}}
 \geq 0,\ \lambda_{{2}}-\beta_{{2}}
 \geq 0,\ -\lambda_{{1}}+\beta_{{1}}
 \geq 0,\ \lambda_{{1}}-\beta_{{1}}+\beta_{{2}}
 \geq 0\},$}
\\
\mathfrak{c}_{3} &= \parbox[t]{0.8\columnwidth}{$\{ \lambda_2 - \beta_2 \geq 0,\ -\lambda_1 + \beta_1 \geq 0,\ -\beta_1 + 2\beta_2 \geq 0, \ \beta_1 - \beta_2 \geq 0, \lambda_1 - \beta_1 + \beta_2 \geq 0 \},$}
\\
\mathfrak{c}_{4} &= \parbox[t]{0.8\columnwidth}{$\{\beta_{{1}}
 \geq 0,\ \lambda_{{2}}-\beta_{{2}}
 \geq 0,\ -\beta_{{1}}+\beta_{{2}}
 \geq 0,\ \lambda_{{1}}-\beta_{{1}}
 \geq 0\},$}
\\
\mathfrak{c}_{5} &= \parbox[t]{0.8\columnwidth}{$\{\beta_{{2}}
 \geq 0,\ \beta_{{1}}-2\beta_{{2}}
 \geq 0,\ \lambda_{{1}}-\beta_{{1}}
 \geq 0,\ \lambda_{{2}}-\beta_{{2}}
 \geq 0\},$}
\\
\mathfrak{c}_{6} &= \parbox[t]{0.8\columnwidth}{$\{  \beta_1 - \beta_2 \geq 0,\ 2\lambda_2 + \beta_1 - 2\beta_2 \geq 0,\ -\beta_1 + 2\beta_2 \geq 0,\ -\lambda_2 + \beta_2 \geq 0,\ \lambda_1 - \beta_1 \geq 0 \},$}
\\
\mathfrak{c}_{7} &= \parbox[t]{0.8\columnwidth}{$\{\lambda_{{1}}-\beta_{{1}}+2\beta_{{2}}
 \geq 0,\ -\lambda_{{1}}+\beta_{{1}}-\beta_{{2}}
 \geq 0,\ \beta_{{1}}-2\beta_{{2}}
 \geq 0,\ \lambda_{{2}}-\beta_{{2}}
 \geq 0\},$}
\\
\mathfrak{c}_{8} &= \parbox[t]{0.8\columnwidth}{$\{-\lambda_{{1}}+\beta_{{1}}
 \geq 0,\ \lambda_{{1}}
 \geq 0,\ \lambda_{{2}}-\beta_{{2}}
 \geq 0,\ -\beta_{{1}}+\beta_{{2}}
 \geq 0\},$}
\\
\mathfrak{c}_{9} &= \parbox[t]{0.8\columnwidth}{$\{
-\lambda_1 + \beta_1 \geq 0,\ 
\lambda_1 - \beta_1 + \beta_2 \geq 0,\ 
\lambda_1 + 2\lambda_2 - \beta_1 \geq 0,\ 
-\lambda_2 + \beta_2 \geq 0,\ 
\beta_1 - 2\beta_2 \geq 0
\},$}
\\
\mathfrak{c}_{10} &= \parbox[t]{0.8\columnwidth}{$\{
\lambda_1 + \lambda_2 - \beta_2 \geq 0,\ 
\lambda_1 + 2\lambda_2 - \beta_1 \geq 0,\ 
2\lambda_2 + \beta_1 - 2\beta_2 \geq 0,\ 
\lambda_1 - \beta_1 + \beta_2 \geq 0,\ 
{-\lambda_1} + \beta_1 \geq 0
\},$}
\\
\mathfrak{c}_{11} &= \parbox[t]{0.8\columnwidth}{$\{2\lambda_{{2}}+\beta_{{1}}-2\beta_{{2}}
 \geq 0,\ -\lambda_{{2}}+\beta_{{2}}
 \geq 0,\ -\beta_{{1}}+\beta_{{2}}
 \geq 0,\ \lambda_{{1}}-\beta_{{1}}
 \geq 0\},$}
\\
\mathfrak{c}_{12} &= \parbox[t]{0.8\columnwidth}{$\{
-\lambda_2 + \beta_2 \geq 0,\ 
\lambda_2 \geq 0,\ 
\beta_1 - 2\beta_2 \geq 0,\ 
\lambda_1 - \beta_1 \geq 0
\},$}
\\
\mathfrak{c}_{13} &= \parbox[t]{0.8\columnwidth}{$\{-\beta_{{1}}+2\beta_{{2}}
 \geq 0,\ -\lambda_{{1}}+\beta_{{1}}-\beta_{{2}}
 \geq 0,\ \lambda_{{1}}
 \geq 0,\ \lambda_{{2}}-\beta_{{2}}
 \geq 0\},$}
\\
\mathfrak{c}_{14} &= \parbox[t]{0.8\columnwidth}{$\{-\lambda_{{1}}+\beta_{{1}}-\beta_{{2}}
 \geq 0,\ \lambda_{{1}}+2\lambda_{{2}}-\beta_{{1}}
 \geq 0,\ -\lambda_{{2}}+\beta_{{2}}
 \geq 0,\ \beta_{{1}}-2\beta_{{2}}
 \geq 0\},$}
\\
\mathfrak{c}_{15} &= \parbox[t]{0.8\columnwidth}{$\{-\beta_{{1}}+\beta_{{2}}
 \geq 0,\ -\lambda_{{2}}+\beta_{{2}}
 \geq 0,\ 2\lambda_{{2}}+\beta_{{1}}-2\beta_{{2}}
 \geq 0,\ \lambda_{{1}}+\lambda_{{2}}-\beta_{{2}}
 \geq 0,\ -\lambda_{{1}}+\beta_{{1}}
 \geq 0\},$}
\\
\mathfrak{c}_{16} &= \parbox[t]{0.8\columnwidth}{$\{
-\lambda_1 - \lambda_2 + \beta_2 \geq 0,\ 
\lambda_1 + 2\lambda_2 - \beta_1 \geq 0,\ 
2\lambda_2 + \beta_1 - 2\beta_2 \geq 0,\ 
\lambda_1 - \beta_1 + \beta_1 \geq 0,\ 
\beta_1 - \beta_2 \geq 0
\},$}
\\
\mathfrak{c}_{17} &= \parbox[t]{0.8\columnwidth}{$\{\lambda_{{2}}+\beta_{{1}}-\beta_{{2}}
 \geq 0,\ -2\lambda_{{2}}-\beta_{{1}}+2\beta_{{2}}
 \geq 0,\ -\beta_{{1}}+\beta_{{2}}
 \geq 0,\ \lambda_{{1}}-\beta_{{1}}
 \geq 0\},$}
\\
\mathfrak{c}_{18} &= \parbox[t]{0.8\columnwidth}{$\{\beta_{{1}}-\beta_{{2}}
 \geq 0,\ \lambda_{{2}}
 \geq 0,\ -2\lambda_{{2}}-\beta_{{1}}+2\beta_{{2}}
 \geq 0,\ \lambda_{{1}}-\beta_{{1}}
 \geq 0\},$}
\\
\mathfrak{c}_{19} &= \parbox[t]{0.8\columnwidth}{$\{-\lambda_{{1}}+\beta_{{1}}-\beta_{{2}}
 \geq 0,\ -\lambda_{{2}}+\beta_{{2}}
 \geq 0,\ -\beta_{{1}}+2\beta_{{2}}
 \geq 0,\ \lambda_{{1}}+2\lambda_{{2}}-\beta_{{1}}
 \geq 0,\ \lambda_{{1}}+\lambda_{{2}}-\beta_{{2}}
 \geq 0\},$}
\\
\mathfrak{c}_{20} &= \parbox[t]{0.8\columnwidth}{$\{\lambda_{{1}}+\lambda_{{2}}-\beta_{{1}}+\beta_{{2}}
 \geq 0,\ -\lambda_{{1}}-2\lambda_{{2}}+\beta_{{1}}
 \geq 0,\ -\lambda_{{1}}+\beta_{{1}}-\beta_{{2}}
 \geq 0,\ \beta_{{1}}-2\beta_{{2}}
 \geq 0\},$}
\\
\mathfrak{c}_{21} &= \parbox[t]{0.8\columnwidth}{$\{\lambda_{{1}}+\lambda_{{2}}-\beta_{{2}}
 \geq 0,\ -\lambda_{{1}}+\beta_{{1}}
 \geq 0,\ -2\lambda_{{2}}-\beta_{{1}}+2\beta_{{2}}
 \geq 0,\ -\beta_{{1}}+\beta_{{2}}
 \geq 0\},$}
\\
\mathfrak{c}_{22} &= \parbox[t]{0.8\columnwidth}{$\{-\lambda_{{1}}-\lambda_{{2}}+\beta_{{2}}
 \geq 0,\ \lambda_{{1}}
 \geq 0,\ 2\lambda_{{2}}+\beta_{{1}}-2\beta_{{2}}
 \geq 0,\ -\beta_{{1}}+\beta_{{2}}
 \geq 0\},$}
\\
\mathfrak{c}_{23} &= \parbox[t]{0.8\columnwidth}{$\{\lambda_{{2}}
 \geq 0,\ -\lambda_{{1}}-2\lambda_{{2}}+\beta_{{1}}
 \geq 0,\ \lambda_{{1}}-\beta_{{1}}+\beta_{{2}}
 \geq 0,\ \beta_{{1}}-2\beta_{{2}}
 \geq 0\},$}
\\
\mathfrak{c}_{24} &= \parbox[t]{0.8\columnwidth}{$\{-2\lambda_{{2}}-\beta_{{1}}+2\beta_{{2}}
 \geq 0,\ \beta_{{1}}-\beta_{{2}}
 \geq 0,\ \lambda_{{1}}+\lambda_{{2}}-\beta_{{2}}
 \geq 0,\ -\lambda_{{1}}+\beta_{{1}}
 \geq 0,\ \lambda_{{1}}+2\lambda_{{2}}-\beta_{{1}}
 \geq 0\},$}
\\
\mathfrak{c}_{25} &= \parbox[t]{0.8\columnwidth}{$\{-\lambda_{{1}}-2\lambda_{{2}}+\beta_{{1}}
 \geq 0,\ -\lambda_{{1}}+\beta_{{1}}-\beta_{{2}}
 \geq 0,\ -\beta_{{1}}+2\beta_{{2}}
 \geq 0,\ \lambda_{{1}}+\lambda_{{2}}-\beta_{{2}}
 \geq 0\},$}
\\
\mathfrak{c}_{26} &= \parbox[t]{0.8\columnwidth}{$\{\lambda_{{1}}
 \geq 0,\ -\lambda_{{1}}-\lambda_{{2}}+\beta_{{2}}
 \geq 0,\ -\lambda_{{1}}+\beta_{{1}}-\beta_{{2}}
 \geq 0,\ \lambda_{{1}}+2\lambda_{{2}}-\beta_{{1}}
 \geq 0\},$}
\\
\mathfrak{c}_{27} &= \parbox[t]{0.8\columnwidth}{$\{\lambda_{{1}}+2\lambda_{{2}}+\beta_{{1}}-2\beta_{{2}}
 \geq 0,\ -\lambda_{{1}}-\lambda_{{2}}+\beta_{{2}}
 \geq 0,\ -2\lambda_{{2}}-\beta_{{1}}+2\beta_{{2}}
 \geq 0,\ -\beta_{{1}}+\beta_{{2}}
 \geq 0\},$}
\\
\mathfrak{c}_{28} &= \parbox[t]{0.8\columnwidth}{$\{-\lambda_{{1}}-2\lambda_{{2}}+\beta_{{1}}
 \geq 0,\ -\beta_{{1}}+2\beta_{{2}}
 \geq 0,\ 2\lambda_{{2}}+\beta_{{1}}-2\beta_{{2}}
 \geq 0,\ \lambda_{{1}}-\beta_{{1}}+\beta_{{2}}
 \geq 0,\ \lambda_{{1}}+\lambda_{{2}}-\beta_{{2}}
 \geq 0\},$}
\\
\mathfrak{c}_{29} &= \parbox[t]{0.8\columnwidth}{$\{\lambda_{{2}}
 \geq 0,\ -\lambda_{{1}}-2\lambda_{{2}}+\beta_{{1}}
 \geq 0,\ -2\lambda_{{2}}-\beta_{{1}}+2\beta_{{2}}
 \geq 0,\ \lambda_{{1}}+\lambda_{{2}}-\beta_{{2}}
 \geq 0\},$}
\\
\mathfrak{c}_{30} &= \parbox[t]{0.8\columnwidth}{$\{-\lambda_{{1}}-\lambda_{{2}}+\beta_{{2}}
 \geq 0,\ \beta_{{1}}-\beta_{{2}}
 \geq 0,\ \lambda_{{1}}+2\lambda_{{2}}-\beta_{{1}}
 \geq 0,\ -2\lambda_{{2}}-\beta_{{1}}+2\beta_{{2}}
 \geq 0\},$}
\\
\mathfrak{c}_{31} &= \parbox[t]{0.8\columnwidth}{$\{-\lambda_{{1}}-2\lambda_{{2}}+\beta_{{1}}
 \geq 0,\ -\lambda_{{1}}+\beta_{{1}}-\beta_{{2}}
 \geq 0,\ 2\lambda_{{1}}+2\lambda_{{2}}-\beta_{{1}}
 \geq 0,\ -\lambda_{{1}}-\lambda_{{2}}+\beta_{{2}}
 \geq 0\},$}
\\
\mathfrak{c}_{32} &= \parbox[t]{0.8\columnwidth}{$\{2\lambda_{{2}}+\beta_{{1}}-2\beta_{{2}}
 \geq 0,\ -\lambda_{{1}}-\lambda_{{2}}+\beta_{{2}}
 \geq 0,\ -\lambda_{{1}}-2\lambda_{{2}}+\beta_{{1}}
 \geq 0,\ \lambda_{{1}}-\beta_{{1}}+\beta_{{2}}
 \geq 0\},$}
\\
\mathfrak{c}_{33} &= \parbox[t]{0.8\columnwidth}{$\{-\lambda_{{1}}-\lambda_{{2}}+\beta_{{2}} 
 \geq 0,\ \lambda_{{1}}+2\lambda_{{2}}-\beta_{{2}}
 \geq 0,\ -\lambda_{{1}}-2\lambda_{{2}}+\beta_{{1}}
 \geq 0,\ -2\lambda_{{2}}-\beta_{{1}}+2\beta_{{2}}
 \geq 0\}.$}
\end{align*}
}

In order to get a feeling for this decomposition of $(\Lambda \otimes_\mathbf{Z} \mathbf{R}) \times (Q \otimes_\mathbf{Z} \mathbf{R})$, consider the intersection of the fan with the affine plane given by $\lambda = (1, 2)$ as indicated in figure \ref{fig:cone-intersection}.
In figure \ref{fig:b2-zerlegung-cartesian}, you find a visualization of the induced decomposition of $Q \otimes_\mathbf{Z} \mathbf{R}$ in Cartesian coordinates with respect to the Killing form.
The highest weight $\omega_1 + 2\omega_2$ corresponds to the upper right corner.
Note that this figure describes the structure of the weight multiplicity function of any module of highest weight $k \lambda$ for $k \in \mathbf{Z}_{>0}$.

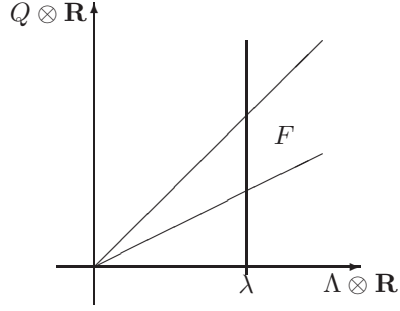
\begin{figure}
\centering
\setlength{\unitlength}{1cm}
\begin{picture}(4,4)(-0.5,-0.5)
\put(-0.5,0){\vector(1,0){4}}
\put(0,-0.5){\vector(0,1){4}}
\put(3.5,-0.1){\makebox(0,0)[t]{$\Lambda \otimes \mathbf{R}$}}
\put(-0.1,3.5){\makebox(0,0)[tr]{$Q \otimes \mathbf{R}$}}
\put(0,0){\line(1,1){3}}
\put(0,0){\line(2,1){3}}
\put(2.5,1.75){\makebox(0,0){$F$}}
\put(2,-0.1){\line(0,1){3.1}}
\put(2,-0.1){\makebox(0,0)[t]{$\lambda$}}
\thicklines
\put(2,1){\line(0,1){1}}
\end{picture}
\caption{Intersecting $F$.}
\label{fig:cone-intersection}
\end{figure}

\begin{figure}
\centering
\input{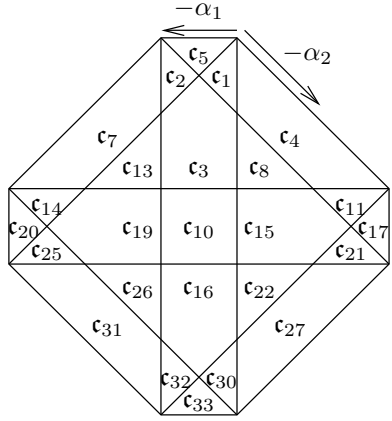}
\caption{Induced decomposition of the Weyl polytope for the highest weight $\lambda = \omega_1 + 2\omega_2$.}
\label{fig:b2-zerlegung-cartesian}
\end{figure}

The quasi-polynomials describing the weight multiplicity function on the above cones are

{
\begin{align*}
f_{1}&=\parbox[t]{0.8\columnwidth}{$\frac{1}{2}\beta_1+\frac{1}{2}\beta_2+\beta_2\beta_1-\frac{1}{2}\beta_2^2+\frac{7}{8}-\frac{1}{4}\beta_1^2+\frac{1}{8}(-1)^{\beta_1}$,}\\
f_{2}&=\parbox[t]{0.8\columnwidth}{$\frac{1}{2}\lambda_1-\frac{1}{2}\beta_1+\frac{3}{2}\beta_2-\frac{1}{4}\lambda_1^2+\frac{1}{2}\beta_2^2+{\frac{7}{8}}-\frac{1}{4}\beta_1^2+\frac{1}{8}(-1)^{\beta_1+\lambda_1}+\frac{1}{2}\lambda_1\beta_1$,}\\
f_{3}&=\parbox[t]{0.8\columnwidth}{$\frac{1}{8}(-1)^{\beta_1}+\frac{1}{2}\lambda_1+\frac{1}{2}\beta_2-\frac{1}{4}\lambda_1^2+\frac{3}{4}-\frac{1}{2}\beta_2^2-\frac{1}{2}\beta_1^2+\frac{1}{2}\lambda_1\beta_1+\frac{1}{8}(-1)^{\beta_1+\lambda_1}+\beta_2\beta_1$,}\\
f_{4}&=\parbox[t]{0.8\columnwidth}{$\frac{1}{8}(-1)^{\beta_1}+\beta_1+\frac{1}{4}\beta_1^2+{\frac{7}{8}}$,}\\
f_{5}&=\parbox[t]{0.8\columnwidth}{$1+\frac{3}{2}\beta_2+\frac{1}{2}\beta_2^2$,}\\
f_{6}&=\parbox[t]{0.8\columnwidth}{$\frac{1}{2}\lambda_2+\frac{1}{2}\beta_1+\beta_2\beta_1-\beta_2^2-\frac{1}{4}\beta_1^2+\lambda_2\beta_2-\frac{1}{2}{\lambda_{
{2}}}^2+\frac{1}{8}(-1)^{\beta_1}+{\frac{7}{8}}$,}\\
f_{7}&=\parbox[t]{0.8\columnwidth}{$\lambda_1-\beta_1+2\beta_2+\frac{1}{4}\lambda_1^2+\frac{1}{8}(-1)^{\beta_1+\lambda_1}+\beta_2^2+{\frac{7}{8}}-\frac{1}{2}\lambda_1\beta_1+\frac{1}{4}\beta_1^2+\lambda_1\beta_2-\beta_2\beta_1$,}\\
f_{8}&=\parbox[t]{0.8\columnwidth}{$\frac{1}{2}\lambda_1+\frac{1}{2}\beta_1-\frac{1}{4}\lambda_1^2+\frac{3}{4}+\frac{1}{2}\lambda_1\beta_1+\frac{1}{8}(-1)^{\beta_1+\lambda_1}+\frac{1}{8}(-1)^{\beta_1}$,}\\
f_{9}&=\parbox[t]{0.8\columnwidth}{$\frac{1}{2}\lambda_2+\frac{1}{2}\lambda_1-\frac{1}{2}\beta_1+\beta_2-\frac{1}{4}\lambda_1^2-\frac{1}{4}\beta_1^2+\lambda_2\beta_2+\frac{1}{2}\lambda_1\beta_1-\frac{1}{2}\lambda_2^2+\frac{1}{8}(-1)^{\beta_1+\lambda_1}+{\frac{7}{8}}$,}\\
f_{10}&=\parbox[t]{0.8\columnwidth}{$\frac{1}{2}\lambda_2+\frac{1}{2}\lambda_1-\frac{1}{4}\lambda_1^2+\frac{3}{4}-\beta_2^2-\frac{1}{2}\beta_1^2+\lambda_2\beta_2+\frac{1}{8}(-1)^{\beta_1}+\frac{1}{2}\lambda_1\beta_1-\frac{1}{2}\lambda_2^2+\frac{1}{8}(-1)^{\beta_1+\lambda_1}+\beta_2\beta_1$,}\\
f_{11}&=\parbox[t]{0.8\columnwidth}{$\frac{1}{2}\lambda_2+\beta_1-\frac{1}{2}\beta_2-\frac{1}{2}\beta_2^2+\frac{1}{4}\beta_1^2+{\frac{7}{8}}+\lambda_2\beta_2+\frac{1}{8}(-1)^{\beta_1}-\frac{1}{2}\lambda_2^2$,}\\
f_{12}&=\parbox[t]{0.8\columnwidth}{$1+\frac{1}{2}\lambda_2+\beta_2+\lambda_2\beta_2-\frac{1}{2}\lambda_2^2$,}\\
f_{13}&=\parbox[t]{0.8\columnwidth}{$\lambda_1-\frac{1}{2}\beta_1+\beta_2+\frac{1}{4}\lambda_1^2+\frac{3}{4}+\lambda_1\beta_2+\frac{1}{8}(-1)^{\beta_1}+\frac{1}{8}(-1)^{\beta_1+\lambda_1}-\frac{1}{2}\lambda_1\beta_1$,}\\
f_{14}&=\parbox[t]{0.8\columnwidth}{$\frac{1}{2}\lambda_2+\lambda_1-\beta_1+\frac{3}{2}\beta_2+\frac{1}{4}\lambda_1^2+\frac{1}{2}\beta_2^2+\lambda_1\beta_2+\frac{1}{4}\beta_1^2+\lambda_2\beta_2-\frac{1}{2}\lambda_2^2-\beta_2\beta_1-\frac{1}{2}\lambda_1\beta_1+{\frac{7}{8}}+\frac{1}{8}(-1)^{\beta_1+\lambda_1}$,}\\
f_{15}&=\parbox[t]{0.8\columnwidth}{$\frac{1}{2}\lambda_2+\frac{1}{2}\lambda_1+\frac{1}{2}\beta_1-\frac{1}{2}\beta_2+\frac{1}{8}(-1)^{\beta_1}-\frac{1}{4}\lambda_1^2+\frac{3}{4}-\frac{1}{2}\beta_2^2+\lambda_2\beta_2-\frac{1}{2}\lambda_2^2+\frac{1}{8}(-1)^{\beta_1+\lambda_1}+\frac{1}{2}\lambda_1\beta_1$,}\\
f_{16}&=\parbox[t]{0.8\columnwidth}{$\lambda_2+\lambda_1-\frac{1}{2}\beta_2+\frac{1}{8}(-1)^{\beta_1}+\frac{1}{4}\lambda_1^2+\frac{3}{4}-\frac{1}{2}\beta_2^2-\lambda_1\beta_2-\frac{1}{2}\beta_1^2+\lambda_2\lambda_1+\frac{1}{2}\lambda_1\beta_1+\frac{1}{8}(-1)^{\beta_1+\lambda_1}+\beta_2\beta_1$,}\\
f_{17}&=\parbox[t]{0.8\columnwidth}{$1+\frac{3}{2}\lambda_2+\frac{3}{2}\beta_1-\frac{3}{2}\beta_2+\frac{1}{2}\beta_2^2+\frac{1}{2}\beta_1^2-\lambda_2\beta_2+\frac{1}{2}\lambda_2^2+\lambda_2\beta_1-\beta_2\beta_1$,}\\
f_{18}&=\parbox[t]{0.8\columnwidth}{$1+\frac{3}{2}\lambda_2+\beta_1-\beta_2-\lambda_2\beta_2+\frac{1}{2}\lambda_2^2+\lambda_2\beta_1$,}\\
f_{19}&=\parbox[t]{0.8\columnwidth}{$\frac{1}{2}\lambda_2+\lambda_1-\frac{1}{2}\beta_1+\frac{1}{2}\beta_2+\frac{1}{4}\lambda_1^2+\frac{3}{4}-\frac{1}{2}\beta_2^2+\lambda_1\beta_2+\lambda_2\beta_2-\frac{1}{2}\lambda_2^2-\frac{1}{2}\lambda_1\beta_1+\frac{1}{8}(-1)^{\beta_1}+\frac{1}{8}(-1)^{\beta_1+\lambda_1}$,}\\
f_{20}&=\parbox[t]{0.8\columnwidth}{$1+\frac{1}{2}\beta_2^2+\frac{1}{2}\lambda_1^2+\frac{1}{2}\lambda_2^2+\lambda_2\lambda_1+\lambda_2\beta_2+\lambda_1\beta_2+\frac{3}{2}\beta_2+\frac{3}{2}\lambda_2+\frac{3}{2}\lambda_1-\frac{3}{2}\beta_1+\frac{1}{2}\beta_1^2-\lambda_2\beta_1-\beta_2\beta_1-\lambda_1\beta_1$,}\\
f_{21}&=\parbox[t]{0.8\columnwidth}{$\frac{3}{2}\lambda_2+\frac{1}{2}\lambda_1+\beta_1-\frac{3}{2}\beta_2-\frac{1}{4}\lambda_1^2+\frac{1}{2}\beta_2^2+\frac{1}{4}\beta_1^2-\lambda_2\beta_2+{\frac{7}{8}}+\frac{1}{2}\lambda_2^2+\frac{1}{8}(-1)^{\beta_1+\lambda_1}+\lambda_2\beta_1+\frac{1}{2}\lambda_1\beta_1-\beta_2\beta_1$,}\\
f_{22}&=\parbox[t]{0.8\columnwidth}{$\lambda_2+\lambda_1+\frac{1}{2}\beta_1-\beta_2+\frac{1}{4}\lambda_1^2+\frac{3}{4}-\lambda_1\beta_2+\lambda_2\lambda_1+\frac{1}{8}(-1)^{\beta_1}+\frac{1}{2}\lambda_1\beta_1+\frac{1}{8}(-1)^{\beta_1+\lambda_1}$,}\\
f_{23}&=\parbox[t]{0.8\columnwidth}{$1+\beta_2+\lambda_1-\beta_1+\frac{1}{2}\lambda_2^2+\lambda_2\lambda_1+\lambda_2\beta_2+\frac{3}{2}\lambda_2-\lambda_2\beta_1$,}\\
f_{24}&=\parbox[t]{0.8\columnwidth}{$\frac{3}{2}\lambda_2+\frac{1}{2}\lambda_1+\frac{1}{2}\beta_1-\beta_2-\frac{1}{4}\lambda_1^2-\frac{1}{4}\beta_1^2-\lambda_2\beta_2+\frac{1}{2}\lambda_2^2+\frac{1}{2}\lambda_1\beta_1+\lambda_2\beta_1+{\frac{7}{8}}+\frac{1}{8}(-1)^{\beta_1+\lambda_1}$,}\\
f_{25}&=\parbox[t]{0.8\columnwidth}{$\frac{3}{2}\lambda_2+\frac{3}{2}\lambda_1-\beta_1+\frac{1}{2}\beta_2+\frac{1}{2}\lambda_1^2-\frac{1}{2}\beta_2^2+\lambda_1\beta_2+\frac{1}{4}\beta_1^2+\lambda_2\lambda_1+\lambda_2\beta_2+\frac{1}{2}\lambda_2^2-\lambda_1\beta_1+{\frac{7}{8}}+\frac{1}{8}(-1)^{\beta_1}-\lambda_2\beta_1$,}\\
f_{26}&=\parbox[t]{0.8\columnwidth}{$\frac{3}{4}+\frac{3}{4}\lambda_1^2+\lambda_2\lambda_1+\lambda_2+\frac{3}{2}\lambda_1-\frac{1}{2}\lambda_1\beta_1+\frac{1}{8}(-1)^{\beta_1}-\frac{1}{2}\beta_1+\frac{1}{8}(-1)^{\beta_1+\lambda_1}$,}\\
f_{27}&=\parbox[t]{0.8\columnwidth}{$2\lambda_2+\lambda_1+\beta_1-2\beta_2+\frac{1}{4}\lambda_1^2+\beta_2^2-\lambda_1\beta_2+\frac{1}{4}\beta_1^2+\lambda_2\lambda_1-2\lambda_2\beta_2+\lambda_2^2+\frac{1}{2}\lambda_1\beta_1+\frac{1}{8}(-1)^{\beta_1+\lambda_1}-\beta_2\beta_1+\lambda_2\beta_1+{\frac{7}{8}}$,}\\
f_{28}&=\parbox[t]{0.8\columnwidth}{$\frac{3}{2}\lambda_2+\lambda_1-\frac{1}{2}\beta_1-\beta_2^2-\frac{1}{4}\beta_1^2+\lambda_2\lambda_1+\lambda_2\beta_2+\frac{1}{2}\lambda_2^2+\beta_2\beta_1-\lambda_2\beta_1+{\frac{7}{8}}+\frac{1}{8}(-1)^{\beta_1}$,}\\
f_{29}&=\parbox[t]{0.8\columnwidth}{$1+\frac{5}{2}\lambda_2+\lambda_1-\beta_2+\lambda_2\lambda_1-\lambda_2\beta_2+\frac{3}{2}\lambda_2^2$,}\\
f_{30}&=\parbox[t]{0.8\columnwidth}{$2\lambda_2+\lambda_1+\frac{1}{2}\beta_1-\frac{3}{2}\beta_2+\frac{1}{4}\lambda_1^2+\frac{1}{2}\beta_2^2-\lambda_1\beta_2-\frac{1}{4}\beta_1^2+\lambda_2\lambda_1-2\lambda_2\beta_2+\lambda_2^2+\frac{1}{2}\lambda_1\beta_1+{\frac{7}{8}}+\frac{1}{8}(-1)^{\beta_1+\lambda_1}+\lambda_2\beta_1$,}\\
f_{31}&=\parbox[t]{0.8\columnwidth}{${\frac{7}{8}}+\lambda_2^2+\lambda_1^2+2\lambda_2\lambda_1+2\lambda_2+2\lambda_1-\beta_1+\frac{1}{4}\beta_1^2-\lambda_2\beta_1-\lambda_1\beta_1+\frac{1}{8}(-1)^{\beta_1}$,}\\
f_{32}&=\parbox[t]{0.8\columnwidth}{$2\lambda_2+\frac{3}{2}\lambda_1-\frac{1}{2}\beta_1-\frac{1}{2}\beta_2+\frac{1}{2}\lambda_1^2-\frac{1}{2}\beta_2^2-\lambda_1\beta_2-\frac{1}{4}\beta_1^2+2\lambda_2\lambda_1+\lambda_2^2+\frac{1}{8}(-1)^{\beta_1}+\beta_2\beta_1+{\frac{7}{8}}-\lambda_2\beta_1$,}\\
f_{33}&=\parbox[t]{0.8\columnwidth}{$1+3\lambda_2+\frac{3}{2}\lambda_1-\frac{3}{2}\beta_2+\frac{1}{2}\lambda_1^2+\frac{1}{2}\beta_2^2-\lambda_1\beta_2+2\lambda_2\lambda_1-2\lambda_2\beta_2+2\lambda_2^2$.}
\end{align*} }

An example on how to use these tables:
In order to determine the character of $V(\lambda)$ for $\lambda = 4 \omega_1 + 8 \omega_2)$ we can observe that the tuples $(\lambda, \beta)$ belong to $\mathfrak{c}_1$ for $\beta \in \{ 0,\ \alpha_1 + \alpha_2,\ 2\alpha_1 + \alpha_2,\ 2\alpha_1 + 2\alpha_2,\ 3\alpha_1 + 2\alpha_2,\ 4\alpha_1 + 2\alpha_2,\ 3\alpha_1 + 3\alpha_2,\ 4\alpha_1 + 3\alpha_2,\ 4\alpha_1 + 4\alpha_2 \}$.
So by evaluating the quasi-polynomial $f_1$, we immediately get the following weight multiplicities:

\begin{align*}
K^{(4,8)}_{(0,0)} &= 1, &
K^{(4,8)}_{(1,1)} &= 2, &
K^{(4,8)}_{(2,1)} &= 3, \\
K^{(4,8)}_{(2,2)} &= 4, &
K^{(4,8)}_{(3,2)} &= 5, &
K^{(4,8)}_{(4,2)} &= 6, \\
K^{(4,8)}_{(3,3)} &= 6, &
K^{(4,8)}_{(4,3)} &= 8, &
K^{(4,8)}_{(4,4)} &= 9.
\end{align*}

Note that if you want to compare the values e.\,g.\ using the LiE online calculator by A.~Cohen et al.%
\footnote{\url{http://www-math.univ-poitiers.fr/~maavl/LiE/form.html}}%
\ you have to take into account that LiE uses the inverse parametrization of the simple roots, and that LiE denotes weights absolutely, not with respect to the highest weight of the module under consideration.
The necessary reparametrization is
\begin{align*}
\tilde\lambda_1 &= \lambda_2, &
\tilde\mu_1 &= \lambda_2 + \beta_1 - 2\beta_2, \\
\tilde\lambda_2 &= \lambda_1, &
\tilde\mu_2 &= \lambda_1 - 2\beta_1 + 2\beta_2.
\end{align*}

\section{Some conclusions}

\begin{corollary}
The weight $0$ does not occur in $V(\lambda)$ unless $\lambda = i \alpha_1 + j \alpha_2$ for nonnegative integers $i, j$ such that $\frac{i}{2} \leq j \leq i$.
In this case \[ \dim V(\lambda)_0 = \frac{i}{2} - i^2 + 3ij - 2j^2 + \frac{3 + (-1)^i}{4}. \]
\end{corollary}

\begin{proof}
The weight $0$ does not occur in $V(\lambda)$ unless $\lambda \in Q$, so suppose $\lambda = i \alpha_1 + j \alpha_2 = (2i-2j)\omega_1 + (-i+2j)\omega_2$.
The inequalities imposed on $i, j$ are equivalent to $\lambda$ being dominant.
We calculate $K^{(2i-2j, -i+2j)}_{(i,j)}$ using the above results:
The vector $(2i-2j, -i+2j, i, j)$ is contained in $\mathfrak{c}_{10}$, so we get $\dim V(\lambda)_0$ by evaluating $f_{10}$ at this vector.
This yields the asserted formula.
\end{proof}

It is well known that $\dim V(\lambda)_\lambda = 1$ for all dominant weights $\lambda$.
But what is $\dim V(\lambda)_{\lambda-\epsilon}$ for some fixed $\epsilon \in Q$?
See figure \ref{fig:around-lambda} for the picture of the Weyl polytope around the highest weight.

\begin{figure}
\centering
\setlength{\unitlength}{1cm}
\begin{picture}(3,2)(-1.5,-1.5)
\put(0,0){\line(-1,0){1.5}}
\put(0,0){\line(-1,-1){1.5}}
\put(0,0){\line(0,-1){1.5}}
\put(0,0){\line(1,-1){1.5}}
\put(0,0){\circle*{0.1}}
\put(-1,0){\circle*{0.1}}
\put(-1,-1){\circle*{0.1}}
\put(0,-1){\circle*{0.1}}
\put(1,-1){\circle*{0.1}}
\put(0,0.0){\makebox(0,0.4)[t]{$\lambda$}}
\put(-1,0.0){\makebox(0,0.4)[t]{$\lambda-\alpha_1$}}
\put(1.2,-1){\makebox(0,0)[l]{$\lambda-\alpha_2$}}
\put(-1.2,-1){\makebox(0,0)[r]{$\lambda-2\alpha_1-\alpha_2$}}
\put(-1.2,-0.5){\makebox(0,0){$\mathfrak{c}_5$}}
\put(-0.5,-1.2){\makebox(0,0){$\mathfrak{c}_1$}}
\put(0.5,-1.2){\makebox(0,0){$\mathfrak{c}_4$}}
\end{picture}
\caption{Around the highest weight.}
\label{fig:around-lambda}
\end{figure}
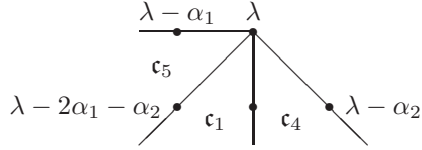

\begin{corollary}
Let $\lambda = \lambda_1 \omega_1 + \lambda_2 \omega_2$ be a dominant weight.
Then the weight multiplicities in $V(\lambda)$ of weights close to $\lambda$ are given by
\begin{align*}
\dim V(\lambda)_{\lambda-\alpha_1} &= 1 \text{ if } \lambda_1 \geq 1, \\
\dim V(\lambda)_{\lambda-2\alpha_1-\alpha_2} &= 3 \text{ if } \lambda_1 \geq 2,\ \lambda_2 \geq 1, \\
\dim V(\lambda)_{\lambda-\alpha_1-\alpha_2} &= 2 \text{ if } \lambda_1, \lambda_2 \geq 1, \\
\dim V(\lambda)_{\lambda-\alpha_2} &= 1 \text{ if } \lambda_2 \geq 1.
\end{align*}
\end{corollary}

\begin{proof}
The first equation can be seen as follows:
$\dim V(\lambda)_{\lambda-\alpha_1} = K^{(\lambda_1,\lambda_2)}_{(1,0)}$.
For $\lambda_1 \geq 1$, the vector $(\lambda_1, \lambda_2, 1, 0)$ is in $\mathfrak{c}_5$.
The value of $f_5$ at this vector is $1$.

The remaining equations can be shown similarly.
Note that in order to show the second and third equation, one can use either $f_5$ or $f_1$ respectively either $f_1$ or $f_4$.
\end{proof}

\section*{Acknowledgements}

This article was prepared during a stay at the Dipartimento di Matematica of the Universit\`a di Roma ``Tor Vergata.''
I thank W.~Baldoni and the dipartimento for their hospitality.



\begin{thebibliography}{14}
 
\bibitem{baldoni2006} W.~Baldoni, M.~Beck, Ch.~Cochet, M.~Vergne, Volume computations for polytopes and partition functions for classical root systems,
\emph{Discrete and computational geometry} \textbf{35} (2006), 551--595.
\mr{2225674}, \doi{10.1007/s00454-006-1234-2}, \arxiv{math/0504231v2}.

\bibitem{billey2004} S.~Billey, V.~Guillemin, E.~Rassart, A vector partition function for the multiplicities of $\mathfrak{sl}_k(\mathbf{C})$, \emph{Journal of algebra} \textbf{278} (2004), 251--293.
\mr{2068078}, \doi{10.1016/j.jalgebra.2003.12.005}, \arxiv{math/0307227v1}.

\bibitem{deconcini2005} C.~De Concini, C.~Procesi, Nested sets and Jeffrey Kirwan residues, pp.~139--149 in: F.~Bogomolov, Y.~Tschinkel (ed.), \emph{Geometric methods in algebra and number theory}, \foreignlanguage{ngerman}{Birkh\"auser}, 2005.
\mr{2159380}, \arxiv{math/0406290v1}.

\bibitem{gelfand1950} I.~Gelfand, M.~Tsetlin, Finite-dimensional representations of the group of unimodular matrices, pp.~653--656 in: I.~Gelfand, \emph{Collected papers}, volume II, Springer, 1988.
Originally appeared in \emph{Doklady akademii nauk SSSR, nowaja serija} \textbf{71} (1950), 825--828.
\zbl{0037.15301}.

\bibitem{heckman1980} G.~Heckman, \emph{Projections of orbits and asymptotic behavior of multiplicities for compact Lie groups}, thesis, Universiteit Leiden, 1980.

\bibitem{jeffrey1995} L.~Jeffrey, F.~Kirwan, Localization for nonabelian group actions, \emph{Topology} \textbf{34} (1995), 291--327. \mr{1318878}, \doi{10.1016/0040-9383(94)00028-J}.

\bibitem{kashiwara1995} M.~Kashiwara, On crystal bases, pp.\ 155--197 in: B.~N.\ Allison et al.\ (ed.), \emph{Representations of groups},
CMS conference proceedings, no.~16, American Mathematical Society, 1995.
\mr{1357199}.

\bibitem{kostant1959} B.~Kostant, A formula for the multiplicity of a weight, \emph{Transactions of the American Mathematical Society} \textbf{93} (1959), 53--73.
\mr{0109192}, \doi{10.2307/1993422}.

\bibitem{littelmann1994} P.~Littelmann, A Littlewood-Richardson rule for symmetrizable Kac-Moody algebras, \emph{Inventiones mathematicae} \textbf{116} (1994), 329--346.
\mr{1253196}, \doi{10.1007/BF01231564}.

\bibitem{littelmann1995} P.~Littelmann, Paths and root operators in representation theory, \emph{Annals of mathematics} \textbf{142} (1995), 499--525.
\mr{1356780}, \doi{10.2307/2118553}.

\bibitem{littelmann1998} P.~Littelmann, Cones, crystals, and patterns, \emph{Transformation groups} \textbf{3} (1998), 145--179.
\mr{1628449}, \doi{10.1007/BF01236431}.

\bibitem{sturmfels1995} B.~Sturmfels, Note on vector partition functions, \emph{Journal of combinatorial theory, series A} \textbf{72} (1995), 302--309.
\mr{1357776}, \doi{10.1016/0097-3165(95)90067-5}.

\bibitem{szenes2003} A.~Szenes, M.~Vergne, Residue formulae for vector partitions and Euler-MacLaurin sums, \emph{Advances in applied mathematics} \textbf{30} (2003), 295--342. \mr{1979797}, \doi{10.1016/S0196-8858(02)00538-9}, \arxiv{math/0202253v1}.

\begin{otherlanguage}{ngerman}
\bibitem{weyl1925f} H.~Weyl, Theorie der Darstellung kontinuierlicher halb\-ein\-fa\-cher Gruppen durch lineare Transformationen, I, II, III, Nachtrag, \emph{Mathematische Zeitschrift} \textbf{23} (1925), 271--309, \textbf{24} (1926), 328--376, 377--395, 789--791.
\zbl{51.0319.01}.
\end{otherlanguage}

\end{thebibliography}
\end{document}